\documentclass[12pt]{amsart}
\usepackage{txfonts}
\usepackage{mathrsfs}
\usepackage{graphicx}
\usepackage{amsfonts}

\usepackage{graphics}
\usepackage{indentfirst}
\usepackage{cite}
\usepackage{latexsym}
\usepackage{amsmath}
\usepackage{amssymb}
\usepackage[dvips]{epsfig}
\usepackage{amscd}

\newtheorem{theorem}{Theorem}[section]
\newtheorem{remark}{Remark}[section]

\newtheorem{lemma}[theorem]{Lemma}

\newcommand{\n}{\rho}

\renewcommand{\div}{ {\rm div }  }

\newcommand{\bt}{\begin{theorem}}
\newcommand{\bl}{\begin{lemma}}
\newcommand{\el}{\end{lemma}}
\newcommand{\et}{\end{theorem}}
\newcommand{\ga}{\gamma}

\newcommand{\la}{\label}

\newcommand{\bn}{\begin{eqnarray}}
\newcommand{\en}{\end{eqnarray}}
\newcommand{\bnn}{\begin{eqnarray*}}
\newcommand{\enn}{\end{eqnarray*}}

\newcommand{\bnnn}{\begin{eqnarray*}}
\newcommand{\ennn}{\end{eqnarray*}}
\newcommand{\ben}{\begin{enumerate}}
\newcommand{\een}{\end{enumerate}}

\newcommand{\ba}{\begin{aligned}}
\newcommand{\ea}{\end{aligned}}
\newcommand{\be}{\begin{equation}}
\newcommand{\ee}{\end{equation}}

\def\O{\mathbb{R}^N}
\def\p{\partial}
\def\norm[#1]#2{\|#2\|_{#1}}

\def\lap{\triangle}

\def\lam{\lambda}

\makeatletter      
\@addtoreset{equation}{section}
\makeatother       

\title[Spherically symmetric flows]
{ Existence and uniqueness of weak solutions of
the compressible spherically symmetric Navier-Stokes equations}

\date{}
\author{Xiangdi H{\small UANG}}

\begin{document}
\maketitle

\begin{abstract}
One of the most influential fundamental tools in harmonic analysis is Riesz transform. It maps $L^p$ functions to $L^p$ functions for any $p\in (1,\infty)$ which plays an important role in singular operators. As an application in fluid dynamics, the norm equivalence between $\|\nabla u\|_{L^p}$ and $\|\mbox{div} u\|_{L^p}+\|\mbox{curl} u\|_{L^p}$ is well established for $p\in (1,\infty)$. However, since Riesz operators sent bounded functions only to BMO functions, there is no hope to bound $\|\nabla u\|_{L^\infty}$ in terms of $\|\mbox{div} u\|_{L^\infty}+\|\mbox{curl} u\|_{L^\infty}$. As pointed out by Hoff[{\it SIAM J. Math. Anal.} {\bf 37}(2006), No. 6, 1742-1760], this is the main obstacle to obtain uniqueness of weak solutions for isentropic compressible flows.

   Fortunately, based on new observations, see Lemma \ref{Riesz}, we derive an exact estimate for $\|\nabla u\|_{L^\infty}\le (2+1/N)\|\mbox{div} u\|_{L^\infty}$ for any N-dimensional radially symmetric vector functions $u$. As a direct application, we give an affirmative answer to the open problem of uniqueness of some weak solutions to the compressible spherically symmetric flows in a bounded ball.
\end{abstract}


\section{Introduction and main results}

We are concerned with the isentropic system of compressible Navier-Stokes equations which reads as
\be\la{n1}
\begin{cases} \rho_t + \div(\rho U) = 0,\\
 (\rho U)_t + \div(\rho U\otimes U) + \nabla P = \mu\lap U + (\mu + \lam)\nabla(\div U),
\end{cases}
\ee
where  $t\ge 0, x\in\Omega\subset\O(N=2,3),\ \rho=\n(t,x)$ and $U=U(t,x)$ are
the density and fluid velocity respectively, and $P=P(\rho)$ is the pressure given by a state equation
\be\la{n2}
P(\rho) = a\rho^{\gamma}
\ee
with the adiabatic constant $\ga>1$ and a positive constant $a$.
The shear viscosity   $\mu$ and the bulk one $\lambda$ are constants satisfying  the physical hypothesis
\be\la{n3}
\mu>0 ,\quad \mu+\frac{N}{2}\lam\ge 0.
\ee
The domain $\Omega$ is a bounded ball with a radius R, namely,
\be
\Omega =B_R=\{x\in\O;\ |x| \le R<\infty\}.
\ee

We study an initial boundary value problem for (1.1) with the
initial condition
\be
(\rho,U)(0,x) =(\rho_0,U_0)(x),\quad x \in \Omega,
\ee
and the boundary condition
\be
U(t,x)=0, \quad t\ge 0,\ x \in \partial \Omega,
\ee
and we are looking for the smooth spherically symmetric solution $(\rho,U)$
of the problem (1.1), (1.5),(1.6) which enjoys the form
\be
\rho(t,x)=\rho(t,|x|),\quad U(t,x)=u(t,|x|)\frac{x}{|x|}.
\ee
Then, for the initial data to be consistent with the form (1.7), we
assume the initial data $(\rho_0,U_0)$ also takes the form
\be\la{bc-3}
\rho_0=\rho_0(|x|),\quad U_0=u_0(|x|)\frac{x}{|x|}.
\ee
In this paper, we further assume the initial density is uniformly positive, that is,
\be\la{bc-3-1}
\rho_0=\rho_0(|x|) \ge \underline{\rho} >0, \quad x \in \Omega
\ee
for a positive constant $\underline{\rho}$.
Then it is noted that as long as the classical solution
of (1.1),(1.5),(1.6) exists the density $\rho$ is positive, that is,
the vacuum never occurs.
It is also noted that since the assumption (1.7) implies
\be
U(t,x) + U(t,-x)=0,\quad x \in \Omega,
\ee
we necessarily have $U(t,0)=0$ (also $U_0(0)=0$).

There are many results about the existence of local and global strong solutions
in time
of the isentropic system of compressible Navier-Stokes equations
when the initial density is uniformly positive (refer to
\cite{Be, Itaya, Ka-1, Ka-2 , Nash, Salvi, Solo, Valli-1, Valli-2}
and their generalization\cite{Mat-1, Mat-2, Mat-3, Tani}
to the full system including the conservation law of energy). On the other hand,
for the initial density allowing vacuum,
the local well-posedness of strong solutions
of the isentropic system was established by Kim\cite{Kim-1}.
For strong solutions with spatial symmetries,
the authors in \cite{Kim-2} proved the global existence
of radially symmetric strong solutions of the isentropic system
in an annular domain, even allowing vacuum initially.
However, it still remains open whether there exist global strong solutions
which are spherically symmetric in a ball. The main difficulties lie on
the lack of estimates of the density and velocity near the center.
In the case vacuum appears, it is worth noting that
Xin\cite{Xin} established a blow-up result which shows
that if the initial density has a compact support, then
any smooth solution to the Cauchy problem of the full system of
compressible Navier-Stokes equations without heat conduction blows up
in a finite time. The same blowup phenomenon occurs also for the isentropic system.
Indeed,  Zhang-Fang (\cite{ZF},Theorem 1.8) showed
that if $(\rho, U) \in C^1([0,T];H^k)\,(k > 3)$
is a spherically symmetric solution to the Cauchy problem with the compact supported initial density,
then the upper limit of $T$ must be finite.
On the other hand, it's unclear whether the strong (classical) solutions
lose their regularity in a finite time
when the initial density is uniformly away from vacuum.

On the other hand, there are amount of literatures investigating the global existence of weak solutions to the compressible Navier-Stokes equations, such as "finite energy solutions" proposed and developed by Lions\cite{L1}, Hoff\cite{Hof1} and Feireisl\cite{F1}, etc. One remarkable result is due to Jiang-Zhang\cite{JZ}, where they prove a global existence for three-dimensional compressible spherically symmetric flows with $\gamma>1$. However, whether their weak solution is unique remains a long-standing open problem. Meanwhile, Desjardins\cite{Des} built a more regular weak solution for three-dimensional torus. Inspired by his work and a new observation in Lemma \ref{Riesz}, we will give some positive answer to the spherically symmetric flow in this paper.

In the spherical coordinates, the original system (\ref{n1}) under the assumption (1.7) takes the form
\be\label{sym}
\left\{
\ba
& \rho_t + (\rho u)_r + (N-1)\frac{\rho u}{r} = 0,\\
& (\rho u)_t + \left(\rho u^2+P(\rho)\right)_r + (N-1)\frac{\rho u^2}{r}
= \kappa\left(u_r+(N-1)\frac{u}{r}\right)_r
\ea
\right.
\ee
where
$
\kappa = 2\mu+\lam.
$
Now, we consider  the following Lagrangian transformation:
\be
t=t,\quad y=\int_0^r\rho(t,s)\, s^{N-1}ds.
\ee
Then, it follows from (1.10) that
\be\la{vip-p}
y_t = -\rho u r^{N-1},\quad r_t = u, \quad r_y = (\rho r^{N-1})^{-1},
\ee
and  the system (\ref{sym}) can be further reduced to
\be\label{sys}
\left\{
\ba
& \rho_t + \rho^2(r^{N-1}u)_y = 0,\\
& r^{1-N}u_t + P_y = \kappa\left(\rho(r^{N-1}u)_y\right)_y
\ea
\right.
\ee
where $t \ge 0$, $y \in [0, M_0]$ and $M_0$ is defined by
\be\la{mass}
M_0=\int_0^R\rho_0(r)\,r^{N-1}dr=\int_0^R\rho(t,r)\,r^{N-1}dr,
\ee
according to the conservation of mass.
Note that
\be
r(t,0)=0,\quad r(t,M_0)=R.
\ee
We denote by $E_0$ the initial energy
\be
E_0 = \int_0^R\left(\rho_0 \frac{u_0^2}{2} + \frac{a\rho_0^{\ga}}{\ga-1}\right)r^{N-1}dr,
\ee
and define a cuboid $Q_{t,y}$, for $t \ge 0$ and $y \in [0, M_0]$, as
\be
Q_{t,y} =[0,t]\times [y,M_0].
\ee
From now on, we denote for $p\in [1,\infty)$ and a radially symmetric function $f=f(|x|)$
\be
\|f\|_{L^p(\Omega)} = (\int_\Omega f^pdx)^{\frac{1}{p}}=\omega_{N}\left(\int_0^Rf^pr^{N-1}dr\right)^{\frac{1}{p}}
\ee
where
\be
\omega_{N}=N|B_1|,\quad\mbox{$|B_1|$ stands for the volume of N-dimensional unit ball.}
\ee

We first prove a local existence of weak solutions to the compressible Navier-Stokes equations in Theorem \ref{t2}.

\noindent
The initial data are supposed to satisfy (\ref{bc-3}-\ref{bc-3-1}) and
\be
\left\{
\ba
& \rho_0\in L^\infty(\Omega)\\
& U_0\in H^1(\Omega)^N.
\ea
\right.
\ee

\begin{theorem}\la{t2}
  Assume $\Omega=B_R$ is a bounded ball in $\mathbb{R}^N$ for $N=2,3$ and $\gamma>1$, then there exists $T_0\in (0,+\infty]$ and a weak solution $(\rho,U)$ to the Navier-Stokes equations (\ref{n1}) in $(0,T_0)$ such that for all $T<T_0$,
  \be\la{regular}
  \left\{
  \ba
  & \rho\in L^\infty((0,T)\times\Omega),\\
  & \rho\dot{U}=\rho(U_t+U\cdot\nabla U),\nabla\dot{U}\in L^2((0,T)\times\Omega)^N,\\
  & div U\in L^2(0,T;L^\infty(\Omega)),\\
  & \nabla U\in L^\infty(0,T;L^2(\Omega))^N\cap L^2(0,T;L^\infty(\Omega))
  \ea
  \right.
  \ee
\end{theorem}

\begin{remark}
  The key idea to establish local existence of weak solution with regularity (\ref{regular}) is to derive uniform upper bound of the density, which  is analogous to Desjardins\cite{Des}. However, there are two obstacles in bounded domain. First of all, due to the lack of commutator estimates, whether there is local weak solution with higher regularity (\ref{regular}) for system (\ref{n1}) remains unknown for initial boundary value problem. We rewrite it in spherically coordinate and derive some new estimates to play a critical role as commutator estimates frequently used by \cite{Des,L1} for Cauchy problem and torus. On the other hand, general global finite energy weak solution with $\gamma>1$ was proved by Jiang-Zhang\cite{JZ}. However, their weak solution is much less regular than (\ref{regular}) and leave a challenging problem on uniqueness of such weak solutions. The main value in Theorem \ref{t2} is to weak the assumption from $\gamma>3$ in \cite{Des} to $\gamma>1$ for three-dimensional spherically symmetric flow, and further more, give an affirmative answer on uniqueness of weak solutions stemming from regularity class (\ref{regular}), which is our main issue in Theorem \ref{t3}. The technical part lies on the combination of Caffarelli-Kohn-Nirenberg\cite{CKN} inequalities with weights and pointwise estimates for radially symmetric functions, see Lemmas \ref{CKN}-\ref{Riesz}.
\end{remark}

\begin{remark}
  From many early works on the blowup criterion \cite{Hxd-1,Hxd-2,Hxd-3,SZ} of strong solutions to the compressible Navier-Stokes equations, the uniform bound of the density induces the regularity of $\rho\dot{U}$, $\nabla\dot{U}$, $div U$ and $curl U$, as indicated by the first three lines in Theorem \ref{t2}. Besides, one of the most important observations is Lemma \ref{Riesz}, which gives desired bound for $\|\nabla U\|_{L^\infty}$ in terms of $\|div U\|_{L^\infty}$. This is a key ingredient in proving the uniqueness of weak solutions illustrated in the following theorem.
\end{remark}

\begin{theorem}\la{t3}
  Let $(\rho^i,U^i)$ be two weak solutions of (1) obtained by Theorem \ref{t2}. Then
  \be
  \rho^1=\rho^2,\quad U^1=U^2\quad \mbox{a.e on $(0,T)\times\Omega$}.
  \ee
\end{theorem}

\begin{remark}
  There are many weak-strong uniqueness results\cite{Ger,Hof2,F2} concerning compressible Navier-Stokes equations. However, weak-weak uniqueness incorporates more difficulties and need special attention. Up to now, the most far reaching result appears in \cite{Hof2}. As pointed out by Hoff\cite{Hof2}, the main obstacle to prevent us from establishing uniqueness is whether we can prove $\nabla U\in L^1L^\infty$ instead of $\nabla U\in L^1BMO$. Such a fact was first verified by Hoff\cite{Hof3}, where the Lipschitz regularity of the velocity $U$ was established with piecewise $C^\alpha$ density. This is a only result concerning uniqueness of weak solutions. With the help of Lemma \ref{Riesz} and Theorem \ref{t2}, we also give an affirmative answer to the spherically symmetric case.

\end{remark}

\section{Proof of Theorem \ref{t2}}

First we recall the following famous Caffarelli-Kohn-Nirenberg\cite{CKN} inequalities with weights.
\begin{lemma}[Caffarelli-Kohn-Nirenberg]\la{CKN}
  There exists a positive constant such that the following inequality holds for all $u\in C_0^\infty(R^n)$
  \be\la{CKN-1}
  \big||x|^{\gamma}u\big|_{L^r}\le C\big||x|^{\alpha}Du\big|_{L^p}^a
  \big||x|^{\beta}u\big|_{L^q}^{1-a}
  \ee
  if and only if the following relations hold:
  \be
  \frac{1}{r}+\frac{\gamma}{n}=a\left(\frac{1}{p}+
  \frac{\alpha-1}{n}\right)+(1-a)\left(\frac{1}{q}+
  \frac{\beta}{n}\right)
  \ee
  \be
  0\le \alpha-\sigma\quad \mbox{if} \ a>0,
  \ee
  and
  \be
  \alpha-\sigma\le 1\quad\mbox{if}\ a>0\quad \mbox{and}\
  \frac{1}{p}+\frac{\alpha-1}{n}=\frac{1}{r}+\frac{\gamma}{n}.
  \ee
  satisfying
  \be
  \ba\la{nnn-1}
  & p,q\ge 1,\ r>0,\ 0\le a\le 1\\
  & \frac{1}{p}+\frac{\alpha-1}{n},\ \frac{1}{q}+\frac{\beta}{n},\ \frac{1}{r}+\frac{\gamma}{n}>0,\\
  &\gamma = a\sigma + (1-a)\beta
  \ea
  \ee
  Furthermore, on any compact set in which (\ref{nnn-1}) and $0\le \alpha-\sigma\le 1$ hold, the constant $C$ is bounded.
\end{lemma}

The following lemma is essential in proving Theorems \ref{t2} and \ref{t3}.

\begin{lemma}\la{Riesz}
Assume $\Omega$ is either a bounded ball $B_R$ with radius $R$ or the whole space $\mathbb{R}^N$, then for any N-dimensional radially symmetric vector functions $U(x)=u(|x|)\frac{x}{|x|}$ for $x\in\Omega$, we have the following estimates
\begin{enumerate}
\item We have
\be
\frac{1}{N}\|div U\|_{L^\infty(\Omega)}\le \|\nabla U\|_{L^\infty(\Omega)}\le (2+\frac{1}{N})\|div U\|_{L^\infty(\Omega)}.
\ee
\item For $p\in [1,\infty)$, we have
\be\la{ur}
\left\{
\ba
&|\frac{u}{r}|\le (\frac{1}{N})^{1-\frac{1}{p}}r^{-\frac{N}{p}}
\omega_{N}^{-\frac{1}{p}}\|div U\|_{L^p(\Omega)}\\
&|u_r|\le \left(1+(N-1)(\frac{1}{N})^{1-\frac{1}{p}}r^{-\frac{N}{p}}
\omega_{N}^{-\frac{1}{p}}\right)\|div U\|_{L^p(\Omega)}
\ea
\right.
\ee
\end{enumerate}
where
\be
\omega_{N,r}=N|B_1|.
\ee
\end{lemma}
\begin{proof} Set $r=|x|$, obviously,
\be\la{vip-1}
div U = u_r + (N-1)\frac{u}{r},\quad curl U=0.
\ee
Denote by
\be\la{vip-2}
u_r + (N-1)\frac{u}{r} = F.
\ee
It follows from (\ref{vip-2}) that
\be\la{vip-3}
(r^{N-1}u)_r = r^{N-1}F,
\ee
which gives
\be
r^{N-1}u = \int_0^rs^{N-1}Fds \le \|F\|_{L^\infty}(\int_0^rs^{N-1}ds)=\frac{r^N}{N}\|F\|_{L^\infty}.
\ee
Consequently,
\be
|\frac{u}{r}|\le \frac{1}{N}\|F\|_{L^\infty}.
\ee
And
\be
\|u_r\|_{L^\infty}\le \|F\|_{L^\infty}+(N-1)\|\frac{u}{r}\|_{L^\infty}
\le \frac{2N-1}{N}\|F\|_{L^\infty}.
\ee
On the other hand, as $r_{x_i}=\frac{x_i}{r}$, we have
\be
\ba
\p_{x_i}(u\frac{x_j}{r})& = u_r\frac{x_ix_j}{r^2} + u(\frac{x_j}{r})_{x_i}\\
& = \frac{x_ix_j}{r^2}u_r + (\delta_{ij}-\frac{x_ix_j}{r^2})\frac{u}{r}.
\ea
\ee
It immediately implies
\be
\|\nabla U\|_{L^\infty(\Omega)}\le \|u_r\|_{L^\infty(\Omega)} + 2\|\frac{u}{r}\|_{L^\infty(\Omega)}\le (2+\frac{1}{N})\|F\|_{L^\infty(\Omega)}.
\ee

For $p\in (1,\infty), r\in(0,\infty)$, one obtains
\be
\ba
r^{N-1}u & = \int_0^rs^{N-1}Fds \\
& \le\left(\int_0^r F^ps^{N-1}ds\right)^{\frac{1}{p}}
\left(\int_0^rs^{N-1}ds\right)^{1-\frac{1}{p}}\\
&\le(\frac{r^N}{N})^{1-\frac{1}{p}}\omega_{N}^{-\frac{1}{p}}
\|F\|_{L^p(\Omega)}
\ea
\ee
Hence,
\be\la{uu-1}
|u|\le (\frac{1}{N})^{1-\frac{1}{p}}r^{1-\frac{N}{p}}
\omega_{N}^{-\frac{1}{p}}\|F\|_{L^p(\Omega)}
\ee
Therefore, (\ref{ur}) follows from (\ref{uu-1}) and (\ref{vip-1}).

This finishes the proof of Lemma.
\end{proof}

We only prove the case when $N=3$  since the case $N=2$ is even simpler.
Throughout of this section, we assume that
$(\rho,U)$ with the form {\rm (1.7)} is the solution
to the initial boundary value problem {\rm (1.1),(1.5),(1.6)} in $[0,T]\times\Omega$,
 and we denote by $C$ generic positive
constants only depending on the initial data and time $T$.

We give a sketch of proof of Theorem \ref{t2}, since the main idea can be borrowed from Desjardins\cite{Des}.
Denote
\be\la{phi}
\Phi(t) = 1+\|\rho\|_{L^\infty(\Omega)} + \|P\|_{L^2(\Omega)}^2 + \|\nabla U\|_{L^2(\Omega)}^2.
\ee
Our main procedure is to derive the following estimate for $t<1$,
\be\la{main-m}
\Phi\le C+C\exp\left(C\exp(\int_0^t\chi(\Phi)\lambda(s)ds)\right).
\ee
for some positive increasing smooth funtion $\chi(x)$ and integrable function $\lambda(s)$.
The existence of $T_0<1$ follows immediately from (\ref{main-m}).

Say denoting by $\zeta(t)$ the right hand side of (\ref{main-m}), we conclude that
\be
\frac{d}{dt}\zeta(t)\le C\pi(\zeta(t))\lambda(t),
\ee
for some smooth function $\pi$. Thus, we have
\be
\int_0^{\zeta(t)}\frac{du}{\pi(u)}\le \int_0^t\lambda(s)ds,
\ee
so that there exists $T_0$ such that for all $T<T_0<1$,
\be
\Phi\le C_T.
\ee

We first have the following basic energy estimate.
Since the proof is standard, we omit it.
\begin{lemma}\la{energy}
It holds for any $0\le t\le T$,
\be\la{en-1}
\int_{\Omega}\left(\rho \frac{|U|^2}{2}+\frac{a\rho^\gamma}{\gamma-1}\right)dx
+ \kappa\int_0^t\int_{\Omega}|\nabla U|^2\,dxd\tau\le \int_{\Omega}\left(\rho_0 \frac{|U_0|^2}{2}
+\frac{a\rho_0^\gamma}{\gamma-1}\right)dx,
\ee
or equivalently,
\be\la{en-2}
\ba
& \int_0^R\left(\rho \frac{u^2}{2}+\frac{a\rho^\gamma}{\gamma-1}\right)r^{N-1}dr +
\kappa\int_0^t\int_0^R\left(u_r^2+\frac{u^2}{r^2}\right)r^{N-1}drd\tau\\
& \le \int_0^R\left(\rho_0 \frac{u_0^2}{2}+
\frac{a\rho_0^\gamma}{\gamma-1}\right)r^{N-1}dr  = E_0.
\ea
\ee
\end{lemma}

 Denoted by
 \be
 G=\kappa div U-P
  \ee
  as effective flux. The momentum equations $(1.1)_2$ can be rewritten as
\be
\rho\dot{U} = \nabla G.
\ee

\medskip
The main difficulty lies in the bound of the density. We will work it in spherical coordinate as system (\ref{sys}) as follows.

\begin{lemma}\la{le-2}
There exists a smooth positive increasing function $\chi(x)$ and an integrable function $\lambda(t)$ such that
\be\label{cru}
\rho(t,y)\le C\exp\left(C\beta(t)+C\exp\left(C\beta(t)+
C\int_0^t\chi(\Phi)\lambda(s)ds\right)
\right),\quad (t,y)\in [0,1]\times[0,M_0].
\ee
where
\be
\beta(t) = \|P\|_{L^2(\Omega)}^2 + \|\nabla U\|_{L^2(\Omega)}^2
\ee
\end{lemma}
\begin{remark}
  From the work of \cite{Des}, refer to (103), we can prove
  \be
  \beta(t)\le C\exp\left(C\int_0^t\chi(\Phi)\lambda(s)ds\right).
  \ee
  Therefore, the left work is concentrated on Lemma \ref{le-2}.
\end{remark}

\begin{proof}\ {\it Step 1.}
In view of (\ref{sys}), it holds
\be\la{ttt}
\ba
\kappa(\log\rho)_{ty}& = \kappa(\frac{\rho_t}{\rho})_y = -\kappa(\rho(r^{N-1}u)_y)_y = -r^{1-N}u_t- p_y\\
& = -(r^{1-N}u)_t - p_y -(N-1)\frac{u^2}{r^N}.
\ea
\ee
Thus, integrating (\ref{ttt}) over $(0,t)\times(0,y)$, we deduce that
\be
\ba
\kappa\log\frac{\rho(t,y)}{\rho(t,0)} &=\kappa \log\frac{\rho_0(y)}{\rho_0(0)} +
\int_{0}^y\left((r^{1-N}u)(0,z)-(r^{1-N}u)(t,z)\right)\,dz\\
& + \int_0^t(p(s,0)-p(s,y))\,ds -\int_0^t\int_{0}^y(N-1)\frac{u^2(s,z)}{r^N}\,dzds,
\ea
\ee
which is equivalent to
\be\la{cru-1}
\ba
\frac{\rho(t,y)}{\rho(t,0)} &=
\frac{\rho_0(y)}{\rho_0(0)}\exp\left(\kappa^{-1}\int_{0}^y((r^{1-N}u)(0,z)-(r^{1-N}u)(t,z))\,dz\right)\\
&\cdot\exp\left(\kappa^{-1}\int_0^t(p(s,0)-p(s,y))\,ds\right)\\
&\cdot\exp\left(-\kappa^{-1}\int_0^t\int_{0}^y(N-1)\frac{u^2(s,z)}{r^N}\,dzds\right).\\
& = \frac{\rho_0(y)}{\rho_0(0)}\Pi_{i=1}^3\Psi_i.
\ea
\ee


 We first deal with $\Psi_1$ and $\Psi_3$.

{\it Step 2.} It follows from Lemma \ref{Riesz} with $N=3$, energy equality (\ref{en-2}) and $\gamma>1$ that
\be
\ba
\int_{0}^{y} r^{1-N}|u|\,dy &=\int_{0}^r\rho|u|\,dr
\le C\|\nabla U\|_{L^2(\Omega)}\int_0^r\rho s^{-\frac{1}{2}}ds\\
&\le C\|\nabla U\|_{L^2(\Omega)}\left(\int_0^r\rho^{6\gamma} s^2ds\right)^{\frac{1}{6\gamma}}
\left(\int_0^rs^{-\frac{3\gamma+2}{6\gamma-1}}ds\right)^{1-\frac{1}{6\gamma}}\\
&\le Cr^{\frac{3(\gamma-1)}{6\gamma}}\|P\|_{L^6(\Omega)}^{\frac{1}{\gamma}}\|\nabla U\|_{L^2(\Omega)}\\
&\le C+C\|P\|_{L^6(\Omega)}^2+C\|\nabla U\|_{L^2(\Omega)}^2\\
& \le C+C\beta(t) + C\int_0^t\|P\|_{\infty}^{12}dt,
\ea
\ee
where we used the following fact.

First recall that
\be
(P^6)_t + div(P^6u) + (6\gamma-1)Pdiv u=0,
\ee
one immediately has
\be
\|P\|_{L^6(\Omega)}^6\le \|P_0\|_{L^6(\Omega)}^6+C\int_0^t\int_{\Omega}|P^6div u|dxds\\
\le C + C\int_0^t\|P\|_{\infty}^{12}ds.
\ee

Hence,
\be\la{psi-1}
\Psi_1,\Psi_1^{-1}\le C\exp\left(C\beta(t)
+C\int_0^t\|P\|_{\infty}^{12}ds\right).
\ee

{\it Step 3.}
Similarly, recall Lemma \ref{Riesz} and $G=\kappa div U-P = \kappa F-P$

\be
\ba
r^2u & = \int_0^rs^{2}Fds \\
& \le\left(\int_0^r F^6s^{4}ds\right)^{\frac{1}{6}}
\left(\int_0^rs^{\frac{8}{5}}ds\right)^{\frac{5}{6}}\\
&\le Cr^{\frac{13}{6}}(\int_0^r F^6s^{4}ds)^{\frac{1}{6}}
\ea
\ee
Hence,
\be
|u|^2/r\le Cr^{-\frac{2}{3}}(\int_0^r F^6s^{4}ds)^{\frac{1}{3}}\le Cr^{-\frac{2}{3}}(\int_0^r G^6s^{4}ds)^{\frac{1}{3}} + Cr\|\rho\|_\infty^{6\gamma}.
\ee
 Also, we will use the following CKN inequality easily from Lemma \ref{CKN}
\be
\|s^{\frac{2}{3}}G\|_{L^6(0,r)}\le C\|s\nabla G\|_{L^2(0,r)}^{\frac{2}{3}}\|sG\|_{L^2(0,r)}^{\frac{1}{3}}
\ee
or equivalently,
\be
\int_0^r G^6s^{4}ds \le C\|\nabla G\|_{L^2(B(0,r))}^{4}\|G\|_{L^2(B(0,r))}^{2}
\ee
Indeed, in view of (\ref{CKN}), take $n=1, \gamma=\frac{2}{3},r=6,\alpha=\beta=1>\sigma=\frac{1}{2},p=q=2,
a=\frac{2}{3}$.

Now we are ready to give estimates for $\Psi_3$.
\be
\ba
\int_0^t\int_{0}^{y}r^{1-N}\frac{|u|^2}{r}\,dyd\tau &
=\int_0^t\int_{0}^r\frac{\rho |u|^2}{s}\,dsd\tau,\quad\mbox{$r\in (0,R)$}\\
&\le C\int_0^t(\int_0^r F^6s^{4}ds)^{\frac{1}{3}} \int_0^r\rho s^{-\frac{2}{3}} dsd\tau\\
&\le C\int_0^tr^{\frac{1}{3}}\|\rho\|_\infty(\int_0^r G^6s^{4}ds)^{\frac{1}{3}} +r^{\frac{1}{3}}\|\rho\|_\infty^{2\gamma+1}d\tau\\
&\le C\int_0^t\left(\|\rho\|_\infty\|G\|_{L^2(B(0,r))}^{\frac{2}{3}} \|G\|_{H^{1}(B(0,r))}^{\frac{4}{3}} + \|\rho\|_\infty^{2\gamma+1}\right)d\tau\\
& = C\int_0^t\left(\|\rho\|_\infty
\|G\|_{L^2(B(0,r))}^{\frac{2}{3}} \|G\|_{H^{1}(B(0,r))}^{\frac{4}{3}} +\|\rho\|_\infty^{2\gamma+1}\right) d\tau\\
&\le C\int_0^t\left(1+\|\rho\|_\infty
\|\nabla U\|_{L^2(B(0,r))}^2 + \|\rho\|_\infty\|\nabla U\|_{L^2}^{\frac{2}{3}} \|\nabla G\|_{L^2(B(0,r))}^{\frac{4}{3}} + \|\rho\|_\infty^{2\gamma+1}\right)d\tau\\
&\le C\int_0^t\left(1+\|\rho\|_{\infty}^{-1}\|\nabla G\|_{L^2(B(0,r))}^2 + \|\nabla U\|_{L^2(B(0,r))}^2(\|\rho\|_{\infty}^5+1) + \|\rho\|_{\infty}^{6\gamma+1}\right)d\tau\\
&\le C\int_0^t\left(1+\|\rho^{\frac{1}{2}}\dot{U}\|_{L^2(B(0,r))}^2 + C(\|\rho\|_\infty^5+1)\|\nabla U\|_{L^2}^2 + \|\rho\|_{\infty}^{6\gamma+1}\right)d\tau
\ea
\ee

The estimates of $\|\rho^{\frac{1}{2}}\dot{U}\|_{L^2}$ then follows similarly as (103) in Desjardins\cite{Des},  one concludes that
\be\la{psi-3}
\Psi_3\le 1\le\Psi_3^{-1}\le C\exp\left(C\int_0^t\chi(\Phi)\lambda(s)ds\right).
\ee

{\it Step 4.} We can rewrite (\ref{cru-1}) as
\be\la{rhoo}
\rho(t,y) = \mathcal{P}(t)\,\mathcal{U}(t,y)\exp\left(-\kappa^{-1}\int_0^tp(s,y)\,ds\right)
\ee
where
\be\la{u-1}
\mathcal{P}(t) = \frac{\rho(t,0)}{\rho_0(0)}\exp\left(\kappa^{-1}\int_0^tp(s,0)\,ds\right)
\ee
and
\be\la{u-2}
\mathcal{U}(t,y) =\rho_0(y)\Psi_1\Psi_3
\ee

On the other hand, it follows from (\ref{rhoo}) that
\be\la{p-1}
\ba
\frac{d}{dt}\exp\left(\frac{\ga}{\kappa}\int_0^tp(s,y)\,ds\right) &
= \frac{a\ga}{\kappa}\rho(t,y)^{\ga}\exp\left(\frac{\ga}{\kappa}\int_0^tp(s,y)\,ds\right)\\
& = \frac{a\ga}{\kappa}\left(\mathcal{P}(t)\,\mathcal{U}(t,y)\right)^{\ga},
\ea
\ee
which implies
\be\la{rhoo-2}
\exp\left(\frac{1}{\kappa}\int_0^tp(s,y)ds\right) =
\left(1+\frac{a\ga}{\kappa}\int_0^t(\mathcal{P}(s)\,\mathcal{U}(s,y))^{\ga}ds\right)^{1/\ga}.
\ee
Next, we are in a position to estimate $\mathcal{P}(t)$. First, observe that
\be\la{my}
\int_{0}^{M_0}\frac{dy}{\rho(t,y)}=\int_{0}^Rr^{N-1}dr = \frac{R^N}{N}.
\ee
In view of (\ref{rhoo}) and (\ref{rhoo-2}), we have
\be\la{rrr}
\rho(t,y) = \frac{\mathcal{P}(t)\,\mathcal{U}(t,y)}
{\left(1+\frac{a\ga}{\kappa}\int_0^t(\mathcal{P}(s)\,\mathcal{U}(s,y))^{\ga}ds\right)^{1/\ga}}.
\ee
Then, $\mathcal{P}(t)$ can be estimated as
\be\la{final}
\ba
&\frac{R^N}{N}\mathcal{P}(t)
= \int_{0}^{M_0}\frac{\mathcal{P}(t)}{\rho(t,y)}\,dy \\
&\quad = \int_{0}^{M_0}
\frac{\left(1+\frac{a\ga}{\kappa}\int_0^t(\mathcal{P}(s)\,\mathcal{U}(s,y))^{\ga}\,ds\right)^{1/\ga}}
{\mathcal{U}(t,y)}\,dy\\
&\quad \le C\int_{0}^{M_0}\frac{1}{\mathcal{U}(t,y)}\,dy \\
&\qquad+ C(\frac{a\ga}{\kappa})^{1/\ga}\left(\sup_{Q_{T,0}} \mathcal{U}(t,y)\right)
\left(\sup_{Q_{T,0}}\mathcal{U}^{-1}(t,y)\right)
\int_{0}^{M_0}\left(\int_0^t\mathcal{P}(s)^\ga ds\right)^{1/\ga}dy\\
&\quad\le CM_0\left(\sup_{Q_{T,0}}\mathcal{U}^{-1}(t,y)\right) \\
&\qquad + CM_0\left(\sup_{Q_{T,0}} \mathcal{U}(t,y)\right)\left(\sup_{Q_{T,0}}
\mathcal{U}^{-1}(t,y)\right)\left(\int_0^t\mathcal{P}(s)^\ga\,ds\right)^{1/\ga}.
\ea
\ee
Using (\ref{my}) and taking $\ga$-th power on both sides of (\ref{final}),
we have
\be\la{Gron}
\ba
\left(\frac{M_0}{E_0}\right)^{\frac{\ga}{\ga-1}}\mathcal{P}(t)^{\ga} & \le C\left(\sup_{Q_{T,0}}
\mathcal{U}^{-1}(t,y)\right)^{\ga} \\+\ & C\left(\sup_{Q_{T,0}}\mathcal{U}^{-1}(t,y)\right)^{\ga}\left(\sup_{Q_{T,0}}
\mathcal{U}(t,y)\right)^{\ga}\left(\int_0^t\mathcal{P}(s)^{\ga}ds\right).
\ea
\ee
Therefore, by Gronwall's inequality, we deduce from (\ref{Gron}) that
\be\la{Gron2}
\ba
\mathcal{P}(t) \le\ & C\left(\frac{E_0}{M_0}\right)^{\frac{1}{\ga-1}}
\left(\sup_{Q_{T,0}}\mathcal{U}^{-1}(t,y)\right) \\
& \cdot\exp \left\{CT\left(\sup_{Q_{T,0}}\mathcal{U}^{-1}(t,y)\right)^{\ga}
\left(\sup_{Q_{T,0}}\mathcal{U}(t,y)\right)^{\ga}\right\}.
\ea
\ee
Finally, recalling (\ref{rrr}), we have
\be\la{density}
\rho(t,y)\le \mathcal{P}(t)\left(\sup_{Q_{T,0}}\mathcal{U}(t,y)\right),
\ee
Plugging (\ref{psi-1}),(\ref{psi-3}) and (\ref{u-2}) into
(\ref{Gron2})-(\ref{density}),
thus finishes the proof of Lemma \ref{le-2}.
\end{proof}

Therefore, Theorem \ref{t2} can be done easily from the work of \cite{Des} and Lemma \ref{Riesz}.

 \bigskip
\section{Proof of Theorem 1.2}

The remaining part is dedicated to the proof of Theorem \ref{t3}.

{\it Proof.} Set $$v^i=\frac{1}{\rho^i}, U^i=u^i\frac{x}{r}.$$
 We prove Theorem \ref{t3} for the system
\be\label{sys-s}
\left\{
\ba
& v_t =(r^{N-1}u)_y,\\
& r^{1-N}u_t + p_y = \kappa\left(\frac{(r^{N-1}u)_y}{v}\right)_y
\ea
\right.
\ee
where $p=av^{-\gamma}$.

The two weak solutions enjoy the following regularity
\be
\left\{
\ba
& 0<\underline{\rho}\le\rho^i\le\overline{\rho},\\
& u^i_r,\frac{u^i}{r}\in L^2L^\infty.
\ea
\right.
\ee
This is equivalent to say
\be\la{equi-2}
\left\{
\ba
& 0<\underline{v}\le v^i\le\overline{v},\\
& r^{N-1}u^i_y,\frac{u^i}{r}\in L^2L^\infty.
\ea
\right.
\ee
Indeed,
\be
r^{N-1}u^i_y = r^{N-1}u^i_rr_y = v^iu^i_r\in L^2L^\infty.
\ee
Denote by
\be
\Lambda = v^1-v^2,\quad, \Theta = u^1-u^2.
\ee
We have
\be
|\Lambda|\le 2\overline{v}, \quad \Lambda(0,y) = \Theta(0,y)=0.
\ee
and
\be\la{lambda-1}
\Lambda_t = (r^{N-1}\Theta)_y,
\ee
\be\la{theta-1}
\ba
r^{1-N}\Theta_t + (p^1-p^2)_y &= \kappa(\frac{r^{N-1}u^1_y}{v^1}-\frac{r^{N-1}u^2_y}{v^2})_y\\
&=\kappa(\frac{(r^{n-1}\Theta)_y}{v^1}-\frac{\Lambda}
{v^1v^2}(r^{N-1}u^2_y))_y.
\ea
\ee

Multiplying (\ref{lambda-1}) by $\Lambda$ and integrating over $(0,y)$, we have
\be\la{lambda-2}
\frac{1}{2}\frac{d}{dt}\int|\Lambda|^2dy \le \epsilon\int|(r^{N-1}\Theta)_y|^2dy + C(\epsilon)\int|\Lambda|^2dy.
\ee
Similarly, multiplying (\ref{theta-1}) by $r^{N-1}\Theta$ and integrating over $(0,y)$, we also obtain
\be\la{theta-2}
\ba
&\frac{1}{2}\frac{d}{dt}\int|\Theta|^2dy + \frac{1}{v^1}\int|(r^{N-1}\Theta)_y|^2dy \\
&\quad\quad\quad\le \epsilon\int|(r^{N-1}\Theta)_y|^2dy + C(\epsilon)\int|\Lambda|^2dy +C(\underline{v},\overline{v})\|r^{N-1}u^2_y\|_{L^\infty}^2\int|\Lambda|^2dy.
\ea
\ee
where we have used Young's inequality and
\be
|p^1-p^2|=|(v^1)^{-\gamma}-(v^2)^{-\gamma}|\le C(\underline{v},\overline{v})|\Lambda|.
\ee
Recalling the fact (\ref{equi-2}) and (\ref{vip-p}) that
\be
(r^{N-1}u^2)_y = r^{N-1}u^2_y + (N-1)v\frac{u^2}{r}\in L^2L^\infty.
\ee
Collecting (\ref{lambda-2}-\ref{theta-2}) and choosing $\epsilon$ small enough, Gronwall's inequality immediately imply
\be
\Lambda = \Theta = 0.
\ee


{\bf Acknowledgment}. \  This research   is supported  by
President Fund of Academy of Mathematics Systems Science, CAS, No.2014-cjrwlzx-hxd and National Natural Science Foundation of China, No. 11471321 and 11371064.

\medskip

\noindent
Xiangdi Huang\\
NCMIS, AMSS,   Chinese Academy of Sciences, Beijing\\
100190, People's Republic of China\\
e-mail: xdhuang@amss.ac.cn

\medskip

\end{document}